\documentclass[]{siamltex}
\usepackage{yuyuan}
\usepackage[colorlinks=true, citecolor=blue]{hyperref}
\usepackage{algorithm, algpseudocode, color, tabularx, multirow, graphicx,placeins,caption,subcaption,epstopdf}
\usepackage{mathtools}
\mathtoolsset{showonlyrefs}

\def\zl{\underline{z}}
\def\zu{\overline{z}}

\def\tw{\tilde{w}}

\def\tz{\tilde{z}}

\def\kp{\tp[k]}

\def\tp{^{t-1}}

\def\kp{_{k-1}}

\title{Mirror-prox sliding methods for solving a class of monotone variational inequalities}
\author{
	Guanghui Lan
	\thanks{H. Milton Stewart School of Industrial and Systems Engineering, Georgia Institute of Technology ({\tt george.lan@isye.gatech.edu}).}
	\and
	Yuyuan Ouyang
	\thanks{School of Mathematical and Statistical Sciences, Clemson University ({\tt yuyuano@clemson.edu}).} 
}

\begin{document}
	
	\maketitle
	
	\begin{abstract}
		In this paper we propose new algorithms for solving a class of structured monotone variational inequality (VI) problems over compact feasible sets. By identifying the gradient components existing in the operator of VI,
		we show that it is possible to skip computations of the gradients from time to time, while still maintaining the optimal iteration complexity for solving these VI problems. Specifically, for deterministic VI problems involving the sum of the gradient of a smooth convex function $\nabla G$ and a monotone operator $H$, we propose a new algorithm, called the mirror-prox sliding method, which is able to compute an $\varepsilon$-approximate weak solution with at most $\cO((L/\varepsilon)^{1/2})$ evaluations of $\nabla G$ and $\cO((L/\varepsilon)^{1/2}+M/\varepsilon)$ evaluations of $H$, where $L$ and $M$ are Lipschitz constants of $\nabla G$ and $H$, respectively. Moreover, for the case when the operator $H$ can only be accessed through its stochastic estimators,
		we propose a stochastic mirror-prox sliding method that can compute a stochastic $\varepsilon$-approximate weak solution with at most $\cO((L/\varepsilon)^{1/2})$ evaluations of $\nabla G$ and $\cO((L/\varepsilon)^{1/2}+M/\varepsilon + \sigma^2/\varepsilon^2)$ samples of $H$, where $\sigma$ is the variance of the stochastic samples of $H$.
	\end{abstract}
	
	\section{Introduction}
	Monotone \emph{variational inequality} (VI) problem has been used to model a broad class of convex optimization, saddle-point (SP) problem, and equilibrium problems. Let $F:\R^n\to\R^n$ be a monotone operator such that (s.t.)
	\begin{align}
		\langle F(u) - F(v), u-v\rangle \ge 0,\ \forall u, v\in\R^n.
	\end{align}
	A VI problem intends to find an $z^*$ in a nonempty closed convex set $Z\subseteq\R^n$ s.t. 
	\begin{align}
		\label{eq:weakVI}
		\langle F(z), z^*-z\rangle \le 0,\ \forall z\in Z.
	\end{align}
	It should be noted that $z^*$ that satisfies \eqref{eq:weakVI} is commonly known as a weak solution. A related notion is a strong VI solution, namely, a solution $z^*\in Z$ s.t.
	\begin{align}
		\label{eq:strongVI}
		\langle F(z^*), z^* - z\rangle\le 0, \forall z\in Z.
	\end{align}
	Note that since $F$ is monotone, a strong solution defined by \eqref{eq:strongVI} is always a weak solution. If in addition $F$ is continuous, then a weak solution is also a strong solution. 
	
	Our problem of interest is a special class of VI problems \eqref{eq:weakVI} in which $Z$ is compact, and
	\begin{align}
		\label{eq:F}
		F(u) = \nabla G(u) + H(u).
	\end{align}
	Here, the operator $F$ is given by the summation of monotone operators $\nabla G$ and $H$, where $\nabla G$ is the gradient of a convex continuously differentiable function.
	We assume that $\nabla G$ and $H$ are $L$-Lipschitz and $M$-Lipschitz respectively with respect to norm $\|\cdot\|$ s.t.
	\begin{align}
		\label{eq:LM}
		\|\nabla G(w) - \nabla G(v)\|_*\le L\|w-v\|\text{ and }\|H(w) - H(v)\|_*\le M\|w-v\|,\ \forall w,v\in Z,
	\end{align}
	where $\|\cdot\|_*$ is the dual norm of $\|\cdot\|$. Note that by the convexity of function $G$, the first inequality above implies that
	\begin{align}
		\label{eq:L}
		0\le G(w) - G(v) - \langle \nabla G(v), w-v\rangle \le \frac{L}{2}\|w-v\|^2,\ \forall w,v\in Z.
	\end{align}
	To accommodate the possible nonsmoothness of the VI operators, we will also introduce a slight generalization of \eqref{eq:weakVI}. Specifically, our goal is to find a weak solution $z^*\in Z$ such that
	\begin{align}
		\label{eq:problem}
		\langle F(z), z^* -z \rangle + J(z^*) - J(z)\le 0,\ \forall z\in Z,
	\end{align}
	where $J$ is a relatively simple convex function. 
	Throughout this paper, we will denote problem \eqref{eq:problem} by $VI(Z;G,H,J)$.

	
	VI has been a classical research area in optimization (see, e.g., \cite{facchinei2003finite} for an extensive description on the research area). One example of our problem of interest $VI(Z;G,H,J)$ is on the game theory of two-player games with nonlinear payoff function:
	\begin{align}
		\label{eq:SPP}
		\min_{x\in X}\max_{y\in Y}f(x) + \langle Kx, y\rangle - g(y).
	\end{align}
	Here $f$ and $g$ are convex differentiable functions and $X$ and $Y$ are convex sets. The above convex-concave saddle point problem is equivalent to $VI(Z;G,H,0)$ in which 
	\begin{align}
		Z = X\times Y,\ G(z)=f(x) + g(y),\ \text{ and }H(z) = \begin{pmatrix}
			Kx\\ -K^\top y
		\end{pmatrix}.
	\end{align}
	See, e.g., \cite{nemirovski2004prox,juditsky2011solving,chen2017accelerated} and the references within for the discussion on solving the above variational inequality problem. See also \cite{ouyang2021lower,zhang2021lower} for iteration complexity lower bounds of first-order methods for solving the above problem. 
	During the past few years, the study of VI has attracted much interest in statistics, machine learning and artificial intelligence due to its close relation with Generalized Linear Models (e.g., \cite{juditsky2020statistical}), Generative Adversary Networks (e.g., \cite{lin2018solving}), Reinforcement Learning (e.g., \cite{lan2021policy}) and others. 
	Some classic numerical methods for solving VI include deterministic/stochastic gradient projection (e.g., \cite{facchinei2003finite,nemirovski2009robust}), extragradient/mirror-prox methods (e.g., \cite{korpelevich1976extragradient,nemirovski2004prox,juditsky2011solving,monteiro2010complexity}), proximal-point methods (e.g., \cite{rockafellar1976monotone}), and many others. 
	Much recent research effort on the development of VI algorithms
	has been devoted to their iteration complexity, namely, the number of evaluations of deterministic or stochastic VI operators in order to compute an approximate weak or strong solution. 
	
	A seminal result for solving large-scale VI by 
	Nemirovski \cite{nemirovski2004prox}
	shows that
	to compute an approximate $\varepsilon$-approximate weak solution, the number of monotone operator evaluations performed by a generalized extragradient method, called the mirror-prox method,
	can bounded by $\cO(1/\varepsilon)$, where the constant depends on the Lipschitz constant of the monotone operator and the diameter of the feasible set. 
	Nemirovski's mirror-prox method has inspired many more research efforts on solving large-scale VIs; see, e.g., \cite{korpelevich1976extragradient,rockafellar1976monotone,chen1999homotopy,nesterov1999homogeneous,solodov1999hybrid,solodov2000inexact,nemirovski2004prox,nesterov2007dual,monteiro2010complexity,juditsky2011solving,jiang2008stochastic,nemirovski2009robust,lan2012optimal,juditsky2011solving,yousefian2013regularized,yousefian2014smoothing,koshal2013regularized,kotsalis2020simplea,kotsalis2020simpleb,chen2017accelerated}. It should be noted that the $\cO(1/\varepsilon)$ complexity for solving deterministic monotone variational inequalities is not improvable (see the discussions in Section 5 of \cite{nemirovski2004prox}) under the assumption of a deterministic oracle model that provides information of the monotone operator. For stochastic problems, a stochastic mirror-prox method is proposed in \cite{juditsky2011solving} that exhibits an $\cO(1/\varepsilon^2)$ sample complexity.

	A key feature of Nemirovski's mirror-prox method is that it maintains two sequences, one for computing approximate solutions and one for computing future iterates. In each iteration, there are two associated steps that resemble gradient descents. The ``gradient descent'' step updates approximate solutions, and the ``extragradient'' step updates future iterates. Note that the idea of two associated steps for updating approximate solutions and future iterates is closely related to another seminal work in convex smooth optimization, namely, Nesterov's accelerated gradient method (\cite{nesterov1983method}; see also, e.g., \cite{nesterov2018lectures}). Indeed, one of the simplest versions of Nesterov's accelerated gradient method (see, e.g., Section 2.2, ``Constant Step Scheme I'' in \cite{nesterov2018lectures}) has the feature of two gradient descent steps per iteration, with one gradient step for updating approximate solutions and the other for updating future iterates.
	
	However, although there seems to be a close relation between Nemirovski's mirror-prox method and Nesterov's accelerated gradient method in their principles for the design of  algorithms, 
	there is a gap between the complexities of solving VIs and smooth convex optimization. Specifically, if we consider problem $VI(Z;G,H,J)$ in \eqref{eq:problem}, in the deterministic case it requires $\cO((M+L)/\varepsilon)$ evaluations of both the gradient $\nabla G$ and the monotone operator $H$ for Nemirovski's mirror-prox method in \cite{nemirovski2004prox} to compute an $\varepsilon$-approximate solution. While such complexity of gradient and operator evaluations is later improved to $\cO((L/\varepsilon)^{1/2} + M/\varepsilon)$ in \cite{chen2017accelerated}, it should be noted that the number of gradient evaluations of $\nabla G$ is still in the order $\cO(1/\varepsilon)$ in both \cite{nemirovski2004prox} and \cite{chen2017accelerated}. However, if $H\equiv 0$ and $J\equiv 0$ in $VI(Z;G,H,J)$, it only requires $\cO((L/\varepsilon)^{1/2})$ for Nesterov's accelerated gradient method to compute an approximate solution $\zu$ such that $G(\zu) - G(z)\le \varepsilon$ for all $z\in Z$. Therefore, in terms of the complexity concerning gradient evaluation of $\nabla G$, there still seems to be a significant gap of $\cO(1/\varepsilon)$ versus $\cO(1/\varepsilon^{1/2})$ between Nemirovski's mirror-prox method and Nesterov's accelerated gradient method. To the best of our knowledge, such gap has not yet been closed in the literature. Therefore, the key research question in this paper is the following:
	
	\vgap
	
	\emph{For problem $VI(Z;G,H,J)$ in \eqref{eq:problem}, does there exist deterministic and/or stochastic methods, such that the total number of gradient evaluations of $\nabla G$ is bounded by $\cO(1/\varepsilon^{1/2})$ when computing an $\varepsilon$-approximate solution?}
	
	\vgap
	
	The above research question has not yet been addressed in the existing literature; see Tables \ref{tabComplexityResults} and \ref{tabComplexityResults_S} for a list of state-of-the-art results concerning the gradient evaluations of $\nabla G$. 
	
	\begin{table}[h]
		\centering
		\renewcommand{\arraystretch}{2}
		\caption{\label{tabComplexityResults}Comparison of complexity results under deterministic setting}
		\begin{tabular}{|c|c|c|c|}
			\hline
			Problem class & Bound on $\nabla G$ eval. & Bound on $H$/$\cH$ eval. & Related work 
			\\\hline
			$VI(Z;G,H,0)$ & $\ds\cO\left(\frac{L+M}{\varepsilon}\right)$ & $\ds\cO\left(\frac{L+M}{\varepsilon}\right)$ & \cite{nemirovski2004prox} ~(see also \cite{auslender2005interior,nesterov2007dual}) 
			\\\hline
			$VI(Z;G,H,J)$ & $\ds\cO\left(\frac{L+M}{\varepsilon}\right)$ & $\ds\cO\left(\frac{L+M}{\varepsilon}\right)$ & \cite{monteiro2011complexity}
			\\\hline
			$VI(Z;G,H,J)$ & $\ds\cO\left(\sqrt{\frac{L}{\varepsilon}}+\frac{M}{\varepsilon} \right)$ & $\ds\cO\left(\sqrt{\frac{L}{\varepsilon}}+\frac{M}{\varepsilon} \right)$ & \cite{chen2017accelerated}
			\\\hline
			$VI(Z;G,H,J)$ & $\ds\cO\left(\sqrt{\frac{L}{\varepsilon}}\right)$ & $\ds\cO\left(\sqrt{\frac{L}{\varepsilon}}+\frac{M}{\varepsilon} \right)$ & {\bf this paper}
			\\\hline
		\end{tabular}
		\\
		\ \\
		\caption{\label{tabComplexityResults_S} Comparison of complexity results under stochastic setting}
		\begin{tabular}{|c|c|c|c|}
			\hline
			$VI(Z;G,H,0)$ &  $\ds\cO\left(\frac{L+M}{\varepsilon}+\frac{\sigma^2}{\varepsilon^2}\right)$ & $\ds\cO\left(\frac{L+M}{\varepsilon}+\frac{\sigma^2}{\varepsilon^2}\right)$ & \cite{juditsky2011solving}
			\\\hline
			$VI(Z;G,H,J)$ & $\ds\cO\left(\sqrt{\frac{L}{\varepsilon}}+\frac{M}{\varepsilon} + \frac{\sigma^2}{\varepsilon^2}\right)$ & $\ds\cO\left(\sqrt{\frac{L}{\varepsilon}}+\frac{M}{\varepsilon} + \frac{\sigma^2}{\varepsilon^2}\right)$ & \cite{chen2017accelerated}
			\\\hline
			$VI(Z;G,H,J)$ & $\ds\cO\left(\sqrt{\frac{L}{\varepsilon}}\right)$ & $\ds\cO\left(\sqrt{\frac{L}{\varepsilon}}+\frac{M}{\varepsilon} + \frac{\sigma^2}{\varepsilon^2}\right)$ & {\bf this paper}
			\\\hline
		\end{tabular}
	\end{table}
	
	In this paper, we provide a positive answer to the above research question. Specifically, we make the following two contributions in this paper. First, for the deterministic case of $VI(Z;G,H,J)$, we propose a novel algorithm, namely the mirror-prox sliding (MPS) algorithm, that is able to compute an $\varepsilon$-approximate weak solution with at most $\cO((L/\varepsilon)^{1/2})$ gradient evaluations of $\nabla G$. Such gradient complexity matches the one of Nesterov's accelerated gradient method for convex smooth optimization. It should be noted that our complexity result does not violate the lower complexity of VIs described in Section 5 of \cite{nemirovski2004prox}; indeed, the total number of monotone operator evaluations of $H$ is bounded by $\cO((L/\varepsilon)^{1/2} + M/\varepsilon)$. As a consequence, we are now able to skip gradient computations of $\nabla G$ from time to time to obtain better gradient complexity, while still maintaining the optimal iteration complexity for solving VI problems.

	Second, for stochastic case of $VI(Z;G,H,J)$ in which the monotone operator $H$ can only be accessed through its stochastic sample $\cH$, we propose a stochastic mirror-prox sliding algorithm that is able to compute an stochastic $\varepsilon$-approximate weak solution with still at most $\cO((L/\varepsilon)^{1/2})$ gradient evaluations of $\nabla G$. The total number of stochastic evaluations of $\cH$ is bounded by $\cO((L/\varepsilon)^{1/2} + M/\varepsilon + \sigma^2/\epsilon^2)$, where $\sigma^2$ is the variance of the stochastic sample operator $\cH$. The comparison of complexity results of this paper with other state-of-the-art results is reported in Tables \ref{tabComplexityResults} and \ref{tabComplexityResults_S}.
	
	It should be noted that our results reveal a separation of complexity based on two oracles concerning the gradient evaluations and VI operator evaluations. In the existing literature concerning iteration complexity theory of first-order methods for solving variational inequalities and saddle point problems, it is usually only assumed that there exists a single oracle that returns all first-order information of inquiry points. See, e.g., previous results in \cite{nemirovski2004prox} for VI \eqref{eq:problem} and in \cite{ouyang2021lower,zhang2021lower} for saddle point problem \eqref{eq:SPP}. However, based on our results, we may assume that there exists a oracle that returns gradient evaluations of $\nabla G$ for any inquiry point, and that there exists one other oracle that returns operator/stochastic operation evaluations of $H$. The complexities for the two oracles should be separated and our proposed methods are able to achieve the optimal complexity for each oracle. Under the separate oracle assumption, we are able to obtain improved upper complexity bound results in terms of the gradient evaluation oracle that are better than the lower complexity bounds in \cite{nemirovski2004prox,ouyang2021lower,zhang2021lower} for variational inequalities and saddle point problems.
	
	\section{The mirror-prox sliding method}
	In this section, we develop a new mirror-prox sliding (MPS) method for solving $VI(Z;G,H,J)$ as shown in Algorithm \ref{alg:MPSVI}, and study its convergence properties. 
	\begin{algorithm}[H]
		\caption{\label{alg:MPSVI}The mirror-prox sliding (MPS) method for solving $VI(Z;G,H,J)$}
		\begin{algorithmic}
			\State Choose $z_0\in Z$ and set $\zu_0=z_0$.
			\For {$k=1,\ldots,N$}
			\State Compute
			\begin{align}
				\label{eq:MPSVI:zl}
				\zl_k & =(1-\gamma_k)\zu_{k-1} + \gamma_k z_{k-1}.
			\end{align} 
			\State Set $z_k^0 = z\kp$. 
			\For {$t=1,\ldots,T_k$}
			\State Compute
			\begin{align}
				\label{eq:MPSVI:tzkt}
				\tz_k^t = & \argmin_{z\in Z}\langle \nabla G(\zl_k) + H(z_k\tp), z\rangle + J(z) + \beta_k V(z\kp, z) + \eta_k^t V(z_k\tp, z),
				\\
				\label{eq:MPSVI:zkt}
				z_k^t = & \argmin_{z\in Z}\langle \nabla G(\zl_k) + H(\tz_k^t), z\rangle + J(z) + \beta_k V(z\kp, z) + \eta_k^t V(z_k\tp, z).
			\end{align}
			\EndFor
			\State Set $z_k = z_k^{T_k}$, $\tz_k = \frac{1}{T_k}\sum_{t=1}^{T_k}\tz_k^t$, and
			\begin{align}
				\label{eq:MPSVI:zuk}
				\zu_k = (1-\gamma_k)\zu\kp + \gamma_k\tz_k.
			\end{align}
			\EndFor 
			\State Output $\zu_N$.
		\end{algorithmic}
	\end{algorithm}
	
	In Algorithm \ref{alg:MPSVI}, the function $V(\cdot,\cdot)$ is called a prox-function associated with set $Z$ and norm $\|\cdot\|$. In particular, for a given strongly convex function $\pi(\cdot)$ with respect to norm $\|\cdot\|$ and strong convexity parameter $1$, 
	we define the prox-function $V$ as 
	\begin{align}
		\label{eq:V}
		V(z,u) = \pi(u) - \pi(z) - \langle \pi(z), u-z\rangle,\ \forall z,u\in Z.
	\end{align}
	The above prox-function is also known as the Bregman divergence \cite{bregman1967relaxation}. Note that by the strong convexity of $\pi(\cdot)$ we have 
	\begin{align}
		\label{eq:Vnorm}
		V(z,u) \ge \frac{1}{2}\|z - u\|^2,\ \forall z, u\in X.
	\end{align}
	One simple example of the prox-function is the Euclidean distance $V(z,u):=(1/2)\|z-u\|_2^2$ associated with norm $\|\cdot\|_2$. 
	
	A few remarks are in place for the proposed MPS method in Algorithm \ref{alg:MPSVI}. First, when $T_k\equiv 1$, the MPS method reduces to a deterministic version of the accelerated mirror-prox method in \cite{chen2017accelerated}. In the case when $T_k\equiv 1$ and $H\equiv 0$, it becomes a version of Nesterov's accelerated gradient method (see, e.g., \cite{nesterov2018lectures}). In the case when $T_k\equiv 1$ and $G\equiv 0$, the MPS method becomes Nemirovski's mirror-prox method in \cite{nemirovski2004prox}. 
	Indeed, the MPS method can be understood as a combination of Nesterov's accelerated gradient method and Nemirovski's mirror-prox method as follows. For iterations $k=1,\ldots,N$, the MPS method computes gradient $\nabla G(\zl_k)$ and runs $T_k$ iterations computing approximate weak solution to a generalized monotone variational inequality subproblem of form 
	$$\langle \nabla G(\zl_k) + H(z), z^*_k - z\rangle + (J(z^*_k) + \beta_k V(x\kp, z^*_k)) - (J(z) + \beta_k V(x\kp, z))\le 0,\ \forall z\in Z.$$
	Second, 
	Within $T_k$ inner iterations of the $k$-th outer iteration of the MPS method, we keep using the gradient value $\nabla G(\zl_k)$ without requesting any more gradient evaluations of $\nabla G$. Equivalently, such design allows us to skip computation of $\nabla G$ from time to time and reduce the complexity associated with gradient evaluations.
	
	In order to evaluate the efficiency of Algorithm \ref{alg:MPSVI}, we define
	\begin{align}
		\label{eq:Q}
		Q(\tz, z):=G(\tz) - G(z) +\langle H(z), \tz - z\rangle + J(\tz) - J(z).
	\end{align}
	Due to the convexity of $G(\cdot)$, if $Q(\tz,z)\le 0$ for all $z\in Z$, then $\tz$ is a weak solution to $VI(Z;G,H)$. Our goal in the remaining part of this section is to estimate the number of gradient and operator evaluations our proposed algorithms require in order to compute $\varepsilon$-approximate weak solution $\zu_N$ such that $\sup_{z\in Z} Q(\zu_N, z)\le \varepsilon$.
	
	\vgap
	
	To analyze the convergence of the proposed MPS method we start with the following lemma concerning the property of $\zl_k$ and $\zu_k$ defined in \eqref{eq:MPSVI:zl} and \eqref{eq:MPSVI:zuk} respectively. 
	
	\vgap
	
	\begin{lemma}
		\label{lem:lin_approx}
		For any $\gamma_k\in [0,1]$, we have
		\begin{align}
			& \left[G(\zu_k) - G(z) + \langle H(z), \zu_k - z\rangle + J(\zu_k) - J(z)\right]
			\\
			& - (1-\gamma_k)\left[G(\zu\kp) - G(z) + \langle H(z), \zu\kp - z\rangle + J(\zu\kp) - J(z)\right]
			\\
			\le & \gamma_k\left[\langle \nabla G(\zl_k) + H(z), \tz_k - z\rangle + J(\tz_k) - J(z) + \frac{L\gamma_k}{2}\|\tz_k - z\kp\|^2\right].
		\end{align}
	\end{lemma}
	\begin{proof}
		From the definitions of $\zl_k$ and $\zu_k$ in \eqref{eq:MPSVI:zl} and \eqref{eq:MPSVI:zuk}, the convexity of $G$ and $J$, and the property \eqref{eq:L} from the $L$-Lipschitz continuity of $\nabla G$, we have
		\begin{align}
			& \left[G(\zu_k) - G(z) + \langle H(z), \zu_k - z\rangle + J(\zu_k) - J(z)\right]
			\\
			& - (1-\gamma_k)\left[G(\zu\kp) - G(z) + \langle H(z), \zu\kp - z\rangle + J(\zu\kp) - J(z)\right]
			\\
			\le & G(\zu_k) - (1-\gamma_k) G(\zu\kp) - \gamma_k G(z) + \gamma_k \left[\langle H(z), \tz_k - z\rangle + J(\tz_k) - J(z)\right]
			\\
			\le & G(\zl_k) + \langle \nabla G(\zl_k), \zu_k - \zl_k\rangle + \frac{L}{2}\|\zu_k - \zl_k\|^2
			\\
			& - (1-\gamma_k)\left[G(\zl_k) + \langle \nabla G(\zl_k), \zu\kp - \zl_k\rangle\right] - \gamma_k \left[ G(\zl_k) + \langle \nabla G(\zl_k), z - \zl_k\rangle\right] 
			\\
			& + \gamma_k \left[\langle H(z), \tz_k - z\rangle + J(\tz_k) - J(z)\right]
			\\
			= & \gamma_k\left[\langle \nabla G(\zl_k) + H(z), \tz_k - z\rangle + J(\tz_k) - J(z) + \frac{L\gamma_k}{2}\|\tz_k - z\kp\|^2\right].
		\end{align}
	\end{proof}
	
	\vgap
	
	We will also need the following lemma concerning the sequences of $\tz_k^t$ and $z_k^t$ computed in the inner iterations of our proposed method.
	
	\vgap
	
	\begin{lemma}
		\label{lem:inner_est}
		If 
		\begin{align}
			\label{eq:cond_etaM}
			M\le \sqrt{(\beta_k + \eta_k^t)\eta_k^t},\ \forall k\ge 1 \text{ and }t\ge 1
		\end{align}
		and
		\begin{align}
			\label{eq:cond_etabeta}
			\eta_k^t\le \beta_k + \eta_k\tp,\ \forall k\ge 1\text{ and }t\ge 2,
		\end{align}
		then
		\begin{align}
			\label{eq:inner_est}
			\begin{aligned}
				& \langle \nabla G(\zl_k) + H(z), \tz_k - z\rangle + J(\tz_k) - J(z)
				\\
				\le & - \beta_k V(z\kp, \tz_k) + \left(\beta_k + \frac{\eta_k^1}{T_k}  \right) V(z\kp, z) - \frac{\beta_k + \eta_k^{T_k} }{T_k} V(z_k, z), \ \forall z\in Z.
			\end{aligned}
		\end{align}
	\end{lemma}
	\begin{proof}
		From the optimality conditions of \eqref{eq:MPSVI:tzkt} and \eqref{eq:MPSVI:zkt}, we have
		\begin{align}
			\label{tmp1}
			\begin{aligned}
				& \langle \nabla G(\zl_k) + H(z_k\tp), \tz_k^t - z\rangle + J(\tz_k^t) - J(z)
				\\
				\le & \beta_k\left(V(z\kp, z) - V(z\kp, \tz_k^t) - V(\tz_k^t, z)\right)
				\\
				& + \eta_k^t \left(V(z_k\tp, z) - V(z_k\tp, \tz_k^t) - V(\tz_k^t, z)\right),\ \forall z\in Z
			\end{aligned}		
		\end{align}
		and
		\begin{align}
			\label{tmp2}
			\begin{aligned}
				& \langle \nabla G(\zl_k) + H(\tz_k^t), z_k^t - z\rangle + J(z_k^t) - J(z)
				\\
				\le & \beta_k\left(V(z\kp, z) - V(z\kp, z_k^t) - V(z_k^t, z)\right)
				\\
				& + \eta_k^t \left(V(z_k\tp, z) - V(z_k\tp, z_k^t) - V(z_k^t, z)\right),\ \forall z\in Z
			\end{aligned}
		\end{align}
		respectively.
		In particular, let $z=z_k^t$ in \eqref{tmp1}, the inequality becomes
		\begin{align}
			\label{tmp3}
			\begin{aligned}
				& \langle \nabla G(\zl_k) + H(z_k\tp), \tz_k^t - z_k^t\rangle + J(\tz_k^t) - J(z^k_t)
				\\
				\le & \beta_k\left(V(z\kp, z^k_t) - V(z\kp, \tz_k^t) - V(\tz_k^t, z^k_t)\right)
				\\
				& + \eta_k^t \left(V(z_k\tp, z^k_t) - V(z_k\tp, \tz_k^t) - V(\tz_k^t, z^k_t)\right).
			\end{aligned}		
		\end{align}
		Adding inequalities \eqref{tmp2} and \eqref{tmp3} and recalling the strong convexity of prox-function $V(\cdot,\cdot)$ in \eqref{eq:Vnorm}, we obtain
		\begin{align}
			& \langle G(\zl_k), \tz_k^t - z\rangle + \langle H(\tz_k^t), z_k^t - z\rangle + \langle H(z_k\tp), \tz_k^t - z_k^t\rangle + J(\tz_k^t) - J(z)
			\\
			\le & \beta_k\left(V(z\kp, z) - V(z_k^t, z) - V(z\kp, \tz_k^t) - V(\tz_k^t, z^k_t)\right)
			\\
			& + \eta_k^t \left(V(z_k\tp, z) - V(z_k^t, z)- V(z_k\tp, \tz_k^t) - V(\tz_k^t, z^k_t)\right)
			\\
			\le & \beta_k V(z\kp, z) - \beta_k V(z\kp, \tz_k^t) + \eta_k^t V(z_k\tp, z) - (\beta_k + \eta_k^t) V(z_k^t, z)
			\\
			& - \frac{1}{2}(\beta_k + \eta_k^t) \|\tz_k^t - z_k^t\|^2 - \frac{\eta_k^t}{2} \|z_k\tp - \tz_k^t\|^2
			,\ \forall z\in Z.
		\end{align}
		In the above relation, note that
		\begin{align}
			& \langle H(\tz_k^t), z_k^t - z\rangle + \langle H(z_k\tp), \tz_k^t - z_k^t\rangle
			\\
			= & \langle H(z_k\tp) - H(\tz_k^t), \tz_k^t - z_k^t\rangle + \langle H(\tz_k^t), \tz_k^t - z\rangle
			\\
			\ge & -\|H(z_k\tp) - H(\tz_k^t)\|_*\| \tz_k^t - z_k^t\| + \langle H(z), \tz_k^t - z\rangle
			\\
			\ge & -M\|z_k\tp - \tz_k^t\|_*\| \tz_k^t - z_k^t\| + \langle H(z), \tz_k^t - z\rangle.
		\end{align}
		Here we use the Cauchy-Schwartz inequality, and the monotonicity and $M$-Lipschitz continuity of $H(\cdot)$. Summarizing the above two relations we have
		\begin{align}
			& \langle G(\zl_k) + H(z), \tz_k^t - z\rangle + J(\tz_k^t) - J(z)
			\\
			\le & \beta_k V(z\kp, z) - \beta_k V(z\kp, \tz_k^t) + \eta_k^t V(z_k\tp, z) - (\beta_k + \eta_k^t) V(z_k^t, z)
			\\
			& - \frac{1}{2}(\beta_k + \eta_k^t) \|\tz_k^t - z_k^t\|^2 - \frac{\eta_k^t}{2} \|z_k\tp - \tz_k^t\|^2 + M\|z_k\tp - \tz_k^t\|_*\| \tz_k^t - z_k^t\|
			\\
			\le & \beta_k V(z\kp, z) - \beta_k V(z\kp, \tz_k^t) + \eta_k^t V(z_k\tp, z) - (\beta_k + \eta_k^t) V(z_k^t, z),
		\end{align}
		where the last inequality is from the relation between $M$ and parameters $\beta_k$ and $\eta_k^t$ in \eqref{eq:cond_etaM}. Taking average of the above relation from $t=1,\ldots,T_k$ and recalling the convexity of functions $J(\cdot)$ and $V(z\kp, \cdot)$ and the definition that $\tz_k=(1/T_k)\sum_{t=1}^{T_k}\tz_k^t$, we have
		\begin{align}
			& \langle G(\zl_k) + H(z), \tz_k - z\rangle + J(\tz_k) - J(z)
			\\
			\le & \beta_k V(z\kp, z) - \beta_k V(z\kp, \tz_k) + \frac{1}{T_k}\sum_{t=1}^{T_k}\eta_k^t V(z\kp, z) - (\beta_k + \eta_k^t) V(z_k^t, z).
		\end{align}
		We conclude \eqref{eq:inner_est} immediately by applying the condition of parameters $\beta_k$ and $\eta_k^t$ in \eqref{eq:cond_etabeta} and recalling that $z_k^0 = z\kp$ and $z_k^{T_k}=z_k$.
	\end{proof}
	
	\vgap
	
	With the help of the above lemma, we are now ready to prove the convergence of Algorithm \ref{alg:MPSVI}. The following quantities will be used throughout this paper:
	\begin{align}
		\label{eq:Gamma}
		\Gamma_k = \begin{cases}
			1 & {k=1},
			\\
			(1-\gamma_k)\Gamma\kp & k>1.
		\end{cases}
	\end{align}
	
	\vgap
	
	\begin{theorem}
		\label{thm:MPS}
		Suppose that conditions \eqref{eq:cond_etaM} and \eqref{eq:cond_etabeta} in Lemma \ref{lem:inner_est} hold, and that
		\begin{align}
			\label{eq:cond_conv}
			\gamma_1 = 1, \gamma_k\in [0,1], \beta_k\ge L\gamma_k,\ \forall k\ge 1, \text{ and }\frac{\gamma_k}{\Gamma_k}\left(\beta_k + \frac{\eta_k^1}{T_k}\right) \le \frac{\gamma\kp(\beta\kp + \eta\kp^{T\kp})}{\Gamma\kp T\kp},\ \forall k\ge 2.
		\end{align}
		We have
		\begin{align}
			\label{eq:MPS}
			\begin{aligned}
				Q(\zu_N, z)\le & \Gamma_N\left(\beta_1 + \frac{\eta_1^1}{T_1}\right)V(z_0, z) - \frac{\gamma_N(\beta_N + \eta_N^{T_N})}{T_N}V(z_N, z),\ \forall z\in Z,
			\end{aligned}
		\end{align}
		where $Q(\cdot,\cdot)$ is defined in \eqref{eq:Q}.
	\end{theorem}
	\begin{proof}
		Combining the results in Lemmas \ref{lem:lin_approx} and \ref{lem:inner_est} and recalling the strong convexity of $V(\cdot,\cdot)$ in  \eqref{eq:Vnorm}, we have
		\begin{align}
			& \left[G(\zu_k) - G(z) + \langle H(z), \zu_k - z\rangle + J(\zu_k) - J(z)\right]
			\\
			& - (1-\gamma_k)\left[G(\zu\kp) - G(z) + \langle H(z), \zu\kp - z\rangle + J(\zu\kp) - J(z)\right]
			\\
			\le & \gamma_k\left[\left(\beta_k + \frac{\eta_k^1}{T_k}  \right) V(z\kp, z) - \frac{\beta_k + \eta_k^{T_k} }{T_k} V(z_k, z)\right].
		\end{align}
		Dividing the above relation by $\Gamma_k$ defined in \eqref{eq:Gamma}, summing from $k=1\ldots,N$, and recalling our assumption of parameters in \eqref{eq:cond_conv}, we conclude \eqref{eq:MPS} immediately.
	\end{proof}
	
	\vgap
	
	In the following corollary, we describe a set of parameters that satisfies the assumptions in the above theorem.
	
	\vgap
	
	\begin{corollary}
		\label{corQRate}
		Suppose that the parameters in the outer iterations of Algorithm \ref{alg:MPSVI} are set to
		\begin{align}
			\label{eq:par}
			\gamma_k = \frac{2}{k+1},\ \beta_k = \frac{2L}{k},\ T_k=\left\lceil\frac{kM}{L}\right\rceil,\text{ and }\eta_k^t = \beta_k(t-1) + \frac{LT_k}{k}.
		\end{align}
		In order to compute an approximate solution $\zu_N$ such that $\sup_{z\in Z}Q(\zu_N, z)\le \varepsilon$, the number of evaluations of gradients $\nabla G(\cdot)$ and operators $H(\cdot)$ are bounded by
		\begin{align}
			\label{eq:N}
			N_{\nabla G}:=\cO\left(\sqrt{\frac{L\Omega_{z_0}}{\varepsilon}}\right)\text{ and }N_{H}:=\cO\left(\frac{M\Omega_{z_0}}{\varepsilon} + \sqrt{\frac{L\Omega_{z_0}}{\varepsilon}}\right),
		\end{align}
		respectively, where
		\begin{align}
			\label{eq:D0}
			\Omega_{z_0}:=\sup_{z\in Z}V(z_0, z).
		\end{align}
	\end{corollary}
	\begin{proof}
		We can clearly see that
		\begin{align}
			\eta_k^t \ge \frac{LT_k}{k}\ge \frac{L}{k}\cdot\frac{kM}{L} = M,
		\end{align}
		so condition \eqref{eq:cond_etaM} holds. It is also straightforward to confirm that condition \eqref{eq:cond_etabeta} holds and that $\gamma_k=1$, $\gamma_k\in [0,1]$, and $\beta_k\ge L\gamma_k$ are all satisfied in condition \eqref{eq:cond_conv}. It suffices to verify the last condition in \eqref{eq:cond_conv} in order to apply Theorem \ref{thm:MPS}. 
		Applying our choice of $\gamma_k$ to the definition of $\Gamma_k$ in \eqref{eq:Gamma} we have $\Gamma_k=2/(k(k+1))$. Therefore, from our parameter setting \eqref{eq:par} we have
		\begin{align}
			\frac{\gamma_k}{\Gamma_k}\left(\beta_k + \frac{\eta_k^1}{T_k}\right) = 3L =  \frac{\gamma\kp(\beta\kp + \eta\kp^{T\kp})}{\Gamma\kp T\kp},
		\end{align}
		and hence all conditions of Theorem \ref{thm:MPS} are satisfied. Applying the theorem and substituting our parameter setting \eqref{eq:par} we have
		\begin{align}
			Q(\zu_N, z)\le & \Gamma_N\left(\beta_1 + \frac{\eta_1^1}{T_1}\right)V(z_0, z) - \frac{\gamma_N(\beta_N + \eta_N^{T_N})}{T_N}V(z_N, z) \le \frac{6L V(z_0, z)}{N(N+1)} \le \frac{6L\Omega_{z_0}}{N^2},\ \forall z\in Z.
		\end{align}
		Therefore, to obtain $\sup_{z\in Z}Q(\zu_N, z)\le \varepsilon$ it suffices to run at most $N_{\nabla G}$ iterations of Algorithm \ref{alg:MPSVI}, in which the gradients $\nabla G(\cdot)$ are evaluated $N_{\nabla G}$ times. Since there are $2T_k$ evaluations of operation $H(\cdot)$ in the $k$-th outer iterations of Algorithm \ref{alg:MPSVI}, the total number of evaluations of $H(\cdot)$ is bounded by
		\begin{align}
			\sum_{k=1}^{N_{\nabla G}}2T_k \le & 2\sum_{k=1}^{N_{\nabla G}}\left(\frac{kM}{L} + 1\right) = \frac{M}{L}N_{\nabla G}(N_{\nabla G}+1) + 2N_{\nabla G}
			\\
			\le & \frac{2M}{L}\left(N_{\nabla G}\right)^2 + 2N_{\nabla G} = \cO\left(\frac{M\Omega_{z_0}}{\varepsilon} + \sqrt{\frac{L\Omega_{z_0}}{\varepsilon}}\right).
		\end{align}
	\end{proof}
	
	\vgap
	
	In view of the results obtained in Theorem~\ref{thm:MPS} and Corollary~\ref{corQRate}, the proposed MPS method is able to compute an $\varepsilon$-approximate weak solution with at most $\cO((L/\varepsilon)^{1/2})$ gradient evaluations of $\nabla G$ and $\cO((L/\varepsilon)^{1/2}+ M/\varepsilon)$ operator evaluations of $H$. The MPS algorithm only requires the input information of the Lipschitz constants $L$ and $M$. Our result  reduces the computational effort required for evaluating $\nabla G$ from $\cO(1/\varepsilon)$ to $\cO((1/\varepsilon)^{1/2})$, which is significant especially when the evaluation of $\nabla G$ is the computational bottleneck of solving problem $VI(Z;G,H,J)$ in \eqref{eq:problem}.
	
	\section{The stochastic mirror-prox sliding method}
	In this section, we propose a stochastic version of the MPS method in Algorithm \ref{alg:MPSVI}. The proposed stochastic mirror-prox method (SMPS) is described in Algorithm~\ref{alg:SMPSVI}. 
	\begin{algorithm}[H]
		\caption{\label{alg:SMPSVI}The stochastic mirror-prox sliding (SMPS) method for solving $VI(Z;G,H,J)$}
		\begin{algorithmic}
			\State Modify \eqref{eq:MPSVI:tzkt} and \eqref{eq:MPSVI:zkt} in Algorithm \ref{alg:MPSVI} to
			\begin{align}
				\label{eq:SMPSVI:tzkt}
				\tz_k^t = & \argmin_{z\in Z}\langle \nabla G(\zl_k) + \cH(z_k\tp; \zeta_k^{2t-1}), z\rangle + J(z) + \beta_k V(z\kp, z) + \eta_k^t V(z_k\tp, z)\text{ and }
				\\
				\label{eq:SMPSVI:zkt}
				z_k^t = & \argmin_{z\in Z}\langle \nabla G(\zl_k) + \cH(\tz_k^t; \zeta_k^{2t}), z\rangle + J(z) + \beta_k V(z\kp, z) + \eta_k^t V(z_k\tp, z)
			\end{align}
			respectively.
		\end{algorithmic}
	\end{algorithm}
	
	The SMPS method in Algorithm \ref{alg:SMPSVI} accesses the operator $H$ through its stochastic samples. Specifically, we assume that when we need the information of $H(z)$ for any $z$, we are able to request an unbiased stochastic sample with bounded variance, namely, $\cH(z; \zeta)$ with $\E_{\zeta}\left[\cH(z; \zeta)\right] = H(z)$ and $\E_{\zeta}\left[\|\cH(z; \zeta) - H(z)\|_*^2\right] \le \sigma^2$. Due to the stochastic setting, our goal in the convergence analysis is to estimate the number of gradient evaluations of $\nabla G$ and stochastic sample evaluations of $\cH$ in order to compute a stochastic $\varepsilon$-approximate solution $\zu_N$ such that 
	\begin{align}
		\E\left[\sup_{z\in Z}Q(\zu_N, z)\right]\le \varepsilon,
	\end{align}
	where $Q(\cdot,\cdot)$ is defined in \eqref{eq:Q}. 
	
	Our tool for analysis is the following lemma, which is a stochastic extension of Lemma \ref{lem:inner_est}. The difference between Lemma \ref{lem:inner_est_S} below and the previous deterministic version Lemma \ref{lem:inner_est} is that we will need to address the propagation of inexactness of stochastic samples of operator $H$. We will introduce the following notations for the inexactness:
	\begin{align}
		\label{eq:Delta}
		\Delta_k^{2t-1}:=\cH(z_k\tp;\zeta_k^{2t-1}) - H(z_k\tp)\text{ and }\Delta_k^{2t}:=\cH(\tz_k^t;\zeta_k^{2t}) - H(\tz_k\tp).
	\end{align}
	We will also define three auxiliary sequences $\{w_k\}$, $\{\tw_k^t\}$ and $\{w_k^t\}$ (where $k=1,\ldots,N$ and $t=0,\ldots,T_k$) to describe the impact of the accumulation of inexactness on the approximate solutions computed by the proposed method. The sequences are defined recursively in the follow way. First, let us set $w_0=z_0$.
	Second, for any $k=1,\ldots,N$, we set $w_k^0=w\kp$ and then compute for all $t=1,\ldots,T_k$,
	\begin{align}
		\label{eq:wkt}
		w_k^t = \argmin_{z\in Z} -\langle \Delta_k^{2t}, z\rangle + \beta_k V(w\kp, z) + \eta_k^t V(w_k\tp, z).
	\end{align}
	Finally, we set $w_k = w_k^{T_k}$. Also, for any $k$ and $t$, we define
	\begin{align}
		\label{eq:twkt}
		\tw_k^t = \frac{\eta_k^t}{\beta_k + \eta_k^t}w_k\tp + \frac{\beta_k}{\beta_k + \eta_k^t}w\kp.
	\end{align}
	We make one observation that will be important for our future analysis. If all $\zeta_k^s$'s for stochastic samples are independently distributed, then for all $t\ge 1$ the above $\tw_k^t$ is independent of $\Delta_k^{2t}$ conditional on the history of all previous stochastic samples since it is a convex combination of $w_k\tp$ and $w\kp$. Consequently, noting from \eqref{eq:SMPSVI:tzkt} that $\tz_k^t$ is also independent of $\Delta_k^{2t}$ conditional on the history, we have
	\begin{align}
		\label{eq:tower}
		\E\left[\langle\Delta_k^{2t}, \tz_k^t - \tw_k^t \rangle\right] = 0,\ \forall k=1,\ldots,N,\ t=1,\ldots,T_k.
	\end{align}
	
	\vgap
	
	\begin{lemma}
		\label{lem:inner_est_S}
		If the parameters of Algorithm \ref{alg:SMPSVI} satisfy
		\begin{align}
			\label{eq:cond_etaM_S}
			M\le \sqrt{\frac{1}{3}(\beta_k + \eta_k^t)\eta_k^t},\ \forall k\ge 1 \text{ and }t\ge 1
		\end{align}
		and
		\begin{align}
			\label{eq:cond_etabeta_S}
			\eta_k^t\le \beta_k + \eta_k\tp,\ \forall k\ge 1\text{ and }t\ge 2,
		\end{align}
		then 
		\begin{align}
			\label{eq:inner_est_S}
			\begin{aligned}
				& \langle G(\zl_k) + H(z), \tz_k - z\rangle + J(\tz_k) - J(z)
				\\
				\le & - \beta_k V(z\kp, \tz_k^t)  + \left(\beta_k + \frac{\eta_k^1}{T_k}\right) (V(z\kp, z) + V(w\kp, z)) - \frac{\beta_k + \eta_k^{T_k}}{T_k} (V(z_k, z) + V(w_k, z))
				\\
				& + \frac{1}{T_k}\sum_{t=1}^{T_k}\frac{1}{2(\beta_k + \eta_k^t)}\left(3\|\Delta_k^{2t-1}\|_*^2 + 4\|\Delta_k^{2t}\|_*^2\right) - \langle \Delta_k^{2t}, \tz_k^t - \tw_k^t\rangle,
			\end{aligned}
		\end{align}
		where $\tw_k^t$ is defined in \eqref{eq:twkt}.
	\end{lemma}
	\begin{proof}
		From the optimality condition of \eqref{eq:SMPSVI:tzkt} we have
		\begin{align}
			\begin{aligned}
				& \langle \nabla G(\zl_k) + \cH(z_k\tp;\zeta_k^{2t-1}), \tz_k^t - z\rangle + J(\tz_k^t) - J(z)
				\\
				\le & \beta_k\left(V(z\kp, z) - V(z\kp, \tz_k^t) - V(\tz_k^t, z)\right)
				\\
				& + \eta_k^t \left(V(z_k\tp, z) - V(z_k\tp, \tz_k^t) - V(\tz_k^t, z)\right),\ \forall z\in Z.
			\end{aligned}		
		\end{align}
		Specially, setting $z=z_k^t$ in the above relation we have
		\begin{align}
			\begin{aligned}
				& \langle \nabla G(\zl_k) + \cH(z_k\tp;\zeta_k^{2t-1}), \tz_k^t - z_k^t\rangle + J(\tz_k^t) - J(z^k_t)
				\\
				\le & \beta_k\left(V(z\kp, z^k_t) - V(z\kp, \tz_k^t) - V(\tz_k^t, z^k_t)\right)
				\\
				& + \eta_k^t \left(V(z_k\tp, z^k_t) - V(z_k\tp, \tz_k^t) - V(\tz_k^t, z^k_t)\right).
			\end{aligned}		
		\end{align}	
		Moreover, from the optimality condition of \eqref{eq:SMPSVI:zkt} we also have
		\begin{align}
			\begin{aligned}
				& \langle \nabla G(\zl_k) + \cH(\tz_k^t;\zeta_k^{2t}), z_k^t - z\rangle + J(z_k^t) - J(z)
				\\
				\le & \beta_k\left(V(z\kp, z) - V(z\kp, z_k^t) - V(z_k^t, z)\right)
				\\
				& + \eta_k^t \left(V(z_k\tp, z) - V(z_k\tp, z_k^t) - V(z_k^t, z)\right),\ \forall z\in Z.
			\end{aligned}
		\end{align}
		Summing the above two inequalities and recalling the strong convexity of prox-function $V(\cdot,\cdot)$ in \eqref{eq:Vnorm}, we obtain
		\begin{align}
			& \langle G(\zl_k), \tz_k^t - z\rangle + \langle \cH(\tz_k^t;\zeta_k^{2t}), z_k^t - z\rangle + \langle \cH(z_k\tp;\zeta_k^{2t-1}), \tz_k^t - z_k^t\rangle + J(\tz_k^t) - J(z)
			\\
			\le & \beta_k\left(V(z\kp, z) - V(z_k^t, z) - V(z\kp, \tz_k^t) - V(\tz_k^t, z^k_t)\right)
			\\
			& + \eta_k^t \left(V(z_k\tp, z) - V(z_k^t, z)- V(z_k\tp, \tz_k^t) - V(\tz_k^t, z^k_t)\right)
			\\
			\le & \beta_k V(z\kp, z) - \beta_k V(z\kp, \tz_k^t) + \eta_k^t V(z_k\tp, z) - (\beta_k + \eta_k^t) V(z_k^t, z)
			\\
			& - \frac{1}{2}(\beta_k + \eta_k^t) \|\tz_k^t - z_k^t\|^2 - \frac{\eta_k^t}{2} \|z_k\tp - \tz_k^t\|^2
			,\ \forall z\in Z.
		\end{align}
		In the above relation, note that
		\begin{align}
			& \langle \cH(\tz_k^t;\zeta_k^{2t}), z_k^t - z\rangle + \langle \cH(z_k\tp;\zeta_k^{2t-1}), \tz_k^t - z_k^t\rangle
			\\
			= & \langle \cH(z_k\tp;\zeta_k^{2t-1}) - \cH(\tz_k^t;\zeta_k^{2t}), \tz_k^t - z_k^t\rangle + \langle H(\tz_k^t), \tz_k^t - z\rangle + \langle \Delta_k^{2t}, \tz_k^t - z\rangle
			\\
			\ge & -\|\cH(z_k\tp;\zeta_k^{2t-1}) - \cH(\tz_k^t;\zeta_k^{2t})\|_*\|\tz_k^t - z_k^t\| + \langle H(z), \tz_k^t - z\rangle + \langle \Delta_k^{2t}, \tz_k^t - z\rangle
			\\
			\ge & -\frac{1}{2(\beta_k + \eta_k^t)}\|\cH(z_k\tp;\zeta_k^{2t-1}) - \cH(\tz_k^t;\zeta_k^{2t})\|_*^2 - \frac{1}{2}(\beta_k + \eta_k^t)\|\tz_k^t - z_k^t\|^2  + \langle H(z), \tz_k^t - z\rangle + \langle \Delta_k^{2t}, \tz_k^t - z\rangle.
		\end{align}
		Here we use the definition of $\Delta_k^{2t}$ in \eqref{eq:Delta}, Cauchy-Schwartz inequality, Young's inequality, and the monotonicity of $H(\cdot)$. Summarizing the above two relations we have
		\begin{align}
			\label{eq:tmp1}
			\begin{aligned}
				& \langle G(\zl_k) + H(z), \tz_k^t - z\rangle + J(\tz_k^t) - J(z)
				\\
				\le & \beta_k V(z\kp, z) - \beta_k V(z\kp, \tz_k^t) + \eta_k^t V(z_k\tp, z) - (\beta_k + \eta_k^t) V(z_k^t, z)
				\\
				& + \frac{1}{2(\beta_k + \eta_k^t)}\|\cH(z_k\tp;\zeta_k^{2t-1}) - \cH(\tz_k^t;\zeta_k^{2t})\|_*^2  - \frac{\eta_k^t}{2} \|z_k\tp - \tz_k^t\|^2 - \langle \Delta_k^{2t}, \tz_k^t - z\rangle.
			\end{aligned}
		\end{align}
		We make a few observations related to the above relation. First, by the optimality condition of \eqref{eq:wkt} we have
		\begin{align}
			& -\langle \Delta_k^{2t}, w_k^t - z\rangle 
			\\
			\le & \beta_k (V(w\kp, z) - V(w\kp, w_k^t) - V(w_k^t,z)) + \eta_k^t(V(w_k\tp, z) - V(w_k\tp, w_k^t) - V(w_k^t,z)),\ \forall z\in Z.
		\end{align}
		Second, by the definition of $\tw_k^t$ in \eqref{eq:twkt}, the Cauchy-Schwartz inequality, Young's inequality, and the strong convexity of prox-function $V(\cdot,\cdot)$ in \eqref{eq:Vnorm} we also have
		\begin{align}
			& - \langle \Delta_k^{2t}, \tw_k^t - w_k^t\rangle 
			\\
			= & -\frac{\eta_k^t}{\beta_k + \eta_k^t}\langle \Delta_k^{2t}, w_k\tp - w_k^t\rangle - \frac{\beta_k}{\beta_k + \eta_k^t}\langle \Delta_k^{2t}, w\kp - w_k^t\rangle
			\\
			\le & \frac{\eta_k^t}{\beta_k + \eta_k^t}\|\Delta_k^{2t}\|_*\|w_k\tp - w_k^t\| + \frac{\beta_k}{\beta_k + \eta_k^t}\|\Delta_k^{2t}\|_*\|w\kp - w_k^t\|
			\\
			\le & \frac{\eta_k^t}{2(\beta_k + \eta_k^t)^2}\|\Delta_k^{2t}\|_*^2 + \frac{\eta_k^t}{2}\|w_k\tp - w_k^t\|_2^2  + \frac{\beta_k}{2(\beta_k + \eta_k^t)^2}\|\Delta_k^{2t}\|_*^2 + \frac{\beta_k}{2}\|w\kp - w_k^t\|_2^2
			\\
			\le & \frac{1}{2(\beta_k + \eta_k^t)}\|\Delta_k^{2t}\|_*^2 + \beta_k V(w\kp, w_k^t) + \eta_k^t V(w_k\tp, w_k^t).
		\end{align}
		Third, by the Cauchy-Schwartz inequality, $M$-Lipschitz continuity of $H(\cdot)$ and the condition \eqref{eq:cond_etaM_S} we have
		\begin{align}
			& \frac{1}{2(\beta_k + \eta_k^t)}\|\cH(z_k\tp;\zeta_k^{2t-1}) - \cH(\tz_k^t;\zeta_k^{2t})\|_*^2  - \frac{\eta_k^t}{2} \|z_k\tp - \tz_k^t\|^2
			\\
			\le & \frac{3}{2(\beta_k + \eta_k^t)}\left(\|\Delta_k^{2t-1}\|_*^2 + \|H(z_k\tp) - H(\tz_k^t)\|_*^2 + \|\Delta_k^{2t}\|_*^2\right)  - \frac{\eta_k^t}{2} \|z_k\tp - \tz_k^t\|^2
			\\
			\le & \frac{3}{2(\beta_k + \eta_k^t)}\left(\|\Delta_k^{2t-1}\|_*^2 + \|\Delta_k^{2t}\|_*^2\right) - \frac{1}{2}\left(\eta_k^t - \frac{3M^2}{\beta_k + \eta_k^t}\right)\|z_k\tp - \tz_k^t\|^2
			\\
			\le & \frac{3}{2(\beta_k + \eta_k^t)}\left(\|\Delta_k^{2t-1}\|_*^2 + \|\Delta_k^{2t}\|_*^2\right).
		\end{align}
		Applying the above three observations to \eqref{eq:tmp1} we have
		\begin{align}
			& \langle G(\zl_k) + H(z), \tz_k^t - z\rangle + J(\tz_k^t) - J(z)
			\\
			\le & \beta_k (V(z\kp, z) + V(w\kp, z)) - \beta_k V(z\kp, \tz_k^t) 
			\\
			& + \eta_k^t (V(z_k\tp, z) + V(w_k\tp, z)) - (\beta_k + \eta_k^t) (V(z_k^t, z) + V(w_k^t, z))
			\\
			& + \frac{1}{2(\beta_k + \eta_k^t)}\left(3\|\Delta_k^{2t-1}\|_*^2 + 4\|\Delta_k^{2t}\|_*^2\right) - \langle \Delta_k^{2t}, \tz_k^t - \tw_k^t\rangle.
		\end{align}
		Taking average of the above relation from $t=1,\ldots,T_k$ and recalling the convexity of functions $J(\cdot)$,  $V(z\kp, \cdot)$ and $V(z\kp, \cdot)$ and the definition that $\tz_k = (1/T_k)\sum_{t=1}^{T_k}\tz_k^t$, we have
		\begin{align}
			& \langle G(\zl_k) + H(z), \tz_k - z\rangle + J(\tz_k) - J(z)
			\\
			\le & \beta_k (V(z\kp, z) + V(w\kp, z)) - \beta_k V(z\kp, \tz_k) 
			\\
			& + \frac{1}{T_k}\sum_{t=1}^{T_k}\eta_k^t (V(z_k\tp, z) + V(w_k\tp, z)) - (\beta_k + \eta_k^t) (V(z_k^t, z) + V(w_k^t, z))
			\\
			& + \frac{1}{T_k}\sum_{t=1}^{T_k}\frac{1}{2(\beta_k + \eta_k^t)}\left(3\|\Delta_k^{2t-1}\|_*^2 + 4\|\Delta_k^{2t}\|_*^2\right)  - \langle \Delta_k^{2t}, \tz_k^t - \tw_k^t\rangle.
		\end{align}
		We conclude our result \eqref{eq:inner_est_S} immediately by applying the condition of parameters $\beta_k$ and $\eta_k^t$ in \eqref{eq:cond_etabeta_S} and recalling that $z_k^0 = z\kp$, $z_k^{T_k}=z_k$, $w_k^0=w\kp$, and $w_k^{T_k} = w_k$.

	\end{proof}

	\vgap
	
	With the help of the above lemma, we are now ready to present in the theorem below the convergence result of the SMPS method. 
	
	\vgap
	
	\begin{theorem}
		\label{thm:SMPS}
		Suppose that $\zeta_k^s$'s are independently distributed random samples and that the stochastic operator $\cH$ satisfies unbiasedness $\E_{\zeta_k^{2t-1}}\left[\cH(z_k\tp; \zeta_k^{2t-1})\right] = H(z_k\tp)$ and $\E_{\zeta_k^{2t}}\left[\cH(\tz_k^t; \zeta_k^{2t})\right]=H(\tz_k^t)$ and bounded variance $\E_{\zeta_k^{2t-1}}\left[\|\cH(z_k\tp; \zeta_k^{2t-1})-H(z_k\tp)\|_*^2\right] \le\sigma^2$ and $\E_{\zeta_k^{2t}}\left[\|\cH(\tz_k^t; \zeta_k^{2t})- H(\tz_k^t)\|_*^2\right]\le\sigma^2$ for all $k=1,\ldots,N$ and $t=1,\ldots,T_k$. 
		If conditions \eqref{eq:cond_etaM_S} and \eqref{eq:cond_etabeta_S} hold, and that
		\begin{align}
			\label{eq:cond_conv_S}
			\gamma_1 = 1, \gamma_k\in [0,1], \beta_k\ge L\gamma_k,\ \forall k\ge 1, \text{ and }\frac{\gamma_k}{\Gamma_k}\left(\beta_k + \frac{\eta_k^1}{T_k}\right) \le \frac{\gamma\kp(\beta\kp + \eta\kp^{T\kp})}{\Gamma\kp T\kp},\ \forall k\ge 2,
		\end{align}
		then we have
		\begin{align}
			\label{eq:MPS_S}
			\begin{aligned}
				& \E\left[\sup_{z\in Z}Q(\zu_N, z)\right] \le 2\Gamma_N\left(\beta_1 + \frac{\eta_1^1}{T_1}\right)\Omega_{z_0} + \Gamma_N\sum_{k=1}^{N}\frac{\gamma_k}{\Gamma_k T_k}\sum_{t=1}^{T_k}\frac{7\sigma^2}{2(\beta_k + \eta_k^t)},\ \forall z\in Z.
			\end{aligned}
		\end{align}
		Here $Q(\cdot,\cdot)$ and $\Omega_{z_0}$ are defined in \eqref{eq:Q} and \eqref{eq:D0} respectively.
	\end{theorem}
	\begin{proof}
		Combining the results in Lemmas \ref{lem:lin_approx} and \ref{lem:inner_est_S} and recalling the strong convexity of $V(\cdot,\cdot)$ in \eqref{eq:Vnorm}, we have
		\begin{align}
			& \left[G(\zu_k) - G(z) + \langle H(z), \zu_k - z\rangle + J(\zu_k) - J(z)\right]
			\\
			& - (1-\gamma_k)\left[G(\zu\kp) - G(z) + \langle H(z), \zu\kp - z\rangle + J(\zu\kp) - J(z)\right]
			\\
			\le & \gamma_k\left[\left(\beta_k + \frac{\eta_k^1}{T_k}  \right) (V(z\kp, z) + V(w\kp, z)) - \frac{\beta_k + \eta_k^{T_k} }{T_k} (V(z_k, z) + V(w_k, z))\right]
			\\
			& + \frac{\gamma_k}{T_k}\sum_{t=1}^{T_k}\frac{1}{2(\beta_k + \eta_k^t)}\left(3\|\Delta_k^{2t-1}\|_*^2 + 4\|\Delta_k^{2t}\|_*^2\right) - \langle \Delta_k^{2t}, \tz_k^t - \tw_k^t\rangle.
		\end{align}
		Dividing the above relation by $\Gamma_k$ defined in \eqref{eq:Gamma}, summing from $k=1\ldots,N$, and recalling our assumption of parameters in \eqref{eq:cond_conv_S} and the definition of $Q(\cdot,\cdot)$ in \eqref{eq:Q}, we have
		\begin{align}
			\frac{1}{\Gamma_N}Q(\zu_N, z) \le & \left(\beta_1 + \frac{\eta_1^1}{T_1}\right)(V(z_0,z) + V(w_0,z)) - \frac{\gamma_N(\beta_N+\eta_N^{T_N})}{\Gamma_N T_N}(V(z_N,z) + V(w_N,z))
			\\
			& + \sum_{k=1}^{N}\frac{\gamma_k}{\Gamma_k T_k}\sum_{t=1}^{T_k}\frac{1}{2(\beta_k + \eta_k^t)}\left(3\|\Delta_k^{2t-1}\|_*^2 + 4\|\Delta_k^{2t}\|_*^2\right) - \langle \Delta_k^{2t}, \tz_k^t - \tw_k^t\rangle.
		\end{align}
		From the above result, recalling that $w_0=z_0$ and using notation $\Omega_{z_0}$ defined in \eqref{eq:D0}, we have
		\begin{align}
			\sup_{z\in Z}Q(\zu_N,z)\le & 2\Gamma_N\left(\beta_1 + \frac{\eta_1^1}{T_1}\right)\Omega_{z_0} + \Gamma_N\sum_{k=1}^{N}\frac{\gamma_k}{\Gamma_k T_k}\sum_{t=1}^{T_k}\frac{1}{2(\beta_k + \eta_k^t)}\left(3\|\Delta_k^{2t-1}\|_*^2 + 4\|\Delta_k^{2t}\|_*^2\right) - \langle \Delta_k^{2t}, \tz_k^t - \tw_k^t\rangle.
		\end{align}
		Taking expectation on both sides of the above inequality and recalling our observation in \eqref{eq:tower} concerning the expectation of the last inner product term above, we conclude our result \eqref{eq:MPS_S}.
	\end{proof}
	
	\vgap 
	
	Comparing the results of the above theorem and its deterministic counterpart in Theorem \ref{thm:MPS}, we can observe that the gap function value/expectation of gap function value can be bounded by a quantity involving $\Omega_{z_0}$, which converges to $0$ in the order of $\cO(L/k^2)$. There is an extra summation in the result \eqref{eq:MPS_S} concerning the variance $\sigma^2$ due to the stochastic setting. However, we will prove in the following corollary that, with proper choice of parameters, the summation in \eqref{eq:MPS_S} is also of order $\cO(L/k^2)$.
	
	\vgap
	
	\begin{corollary}
		\label{corQRate_S}
		Suppose that the parameters in the outer iterations of Algorithm \ref{alg:MPSVI} are set to 
		\begin{align}
			\label{eq:par_S}
			\gamma_k = \frac{2}{k+1},\ \beta_k = \frac{2L}{k},\ T_k=\left\lceil\frac{\sqrt{3}kM}{L} + \frac{Nk^2\sigma^2}{\Omega_{z_0}L^2}\right\rceil,\text{ and }\eta_k^t = \beta_k(t-1) + \frac{LT_k}{k},
		\end{align}
		where $\Omega_{z_0}$ is defined in \eqref{eq:D0}.
		In order to compute an approximate solution $\zu_N$ such that $\sup_{z\in Z}\E\left[Q(\zu_N, z)\right]\le \varepsilon$, the number of evaluations of gradients $\nabla G$ and stochastic operators $\cH$ are bounded by
		\begin{align}
			\label{eq:N_S}
			N_{\nabla G}:=\cO\left(\sqrt{\frac{L\Omega_{z_0}}{\varepsilon}}\right)\text{ and }N_{\cH}:=\cO\left(\frac{M\Omega_{z_0}}{\varepsilon} + \frac{\sigma^2\Omega_{z_0}}{\varepsilon^2} + \sqrt{\frac{L\Omega_{z_0}}{\varepsilon}}\right).
		\end{align}
	\end{corollary}
	\begin{proof}
		Clearly,
		\begin{align}
			\eta_k^t \ge \frac{LT_k}{k}\ge \frac{L}{k}\cdot\frac{\sqrt{3}kM}{L} = \sqrt{3}M,
		\end{align}
		so condition \eqref{eq:cond_etaM_S} holds. It is also straightforward to confirm that conditions \eqref{eq:cond_etabeta_S} and \eqref{eq:cond_conv_S} hold. Applying Theorem \ref{thm:SMPS} and substituting our parameter setting \eqref{eq:par_S} (note that applying our choice of $\gamma_k$ to the definition of $\Gamma_k$ in \eqref{eq:Gamma} we have $\Gamma_k=2/(k(k+1))$) we have
		\begin{align}
			\E\left[\sup_{z\in Z}Q(\zu_N, z)\right] 
			\le & 2\Gamma_N\left(\beta_1 + \frac{\eta_1^1}{T_1}\right)\Omega_{z_0} + \Gamma_N\sum_{k=1}^{N}\frac{\gamma_k}{\Gamma_k T_k}\sum_{t=1}^{T_k}\frac{7\sigma^2}{2(\beta_k + \eta_k^t)}
			\\
			\le & \frac{12L \Omega_{z_0}}{N(N+1)} + \frac{2}{N(N+1)}\sum_{k=1}^{N}\frac{k}{T_k}\sum_{t=1}^{T_k}\frac{7\sigma^2 k}{2LT_k}
			\\
			= & \frac{12L \Omega_{z_0}}{N(N+1)} + \frac{1}{N(N+1)}\sum_{k=1}^{N}\frac{7\sigma^2 k^2}{LT_k}
			\\
			\le & \frac{12L\Omega_{z_0}}{N^2} + \frac{1}{N^2}\sum_{k=1}^{N}\frac{7\sigma^2 k^2}{L}\cdot \frac{\Omega_{z_0} L^2}{Nk^2\sigma^2}
			\\
			= & \frac{19L\Omega_{z_0}}{N^2},\ \forall z\in Z.
		\end{align}
		Therefore, to obtain $\sup_{z\in Z}\E[Q(\zu_N, z)]\le \varepsilon$ it suffices to run at most $N_{\nabla G}$ iterations of Algorithm \ref{alg:MPSVI}, in which the gradients $\nabla G$ are evaluated $N_{\nabla G}$ times. Since $2T_k$ evaluations of stochastic operator $\cH$ are performed in the $k$-th iteration of Algorithm \ref{alg:MPSVI}, the total number of stochastic operator evaluations of $\cH$ is bounded by
		\begin{align}
			\sum_{k=1}^{N_{\nabla G}}2T_k \le & 2\sum_{k=1}^{N_{\nabla G}}\left(\frac{\sqrt{3}kM}{L} + \frac{N_{\nabla G}k^2\sigma^2}{\Omega_{z_0}L^2} + 1\right) 
			\\
			\le & \frac{2\sqrt{3}M}{L}\left(N_{\nabla G}\right)^2 + \frac{8\sigma^2}{3\Omega_{z_0} L^2}N_{\nabla G}^4 + 2N_{\nabla G} 
			\\
			= & \cO\left(\frac{M\Omega_{z_0}}{\varepsilon} + \frac{\sigma^2\Omega_{z_0}}{\varepsilon^2} + \sqrt{\frac{L\Omega_{z_0}}{\varepsilon}}\right).
		\end{align}
	\end{proof}
	
	\vgap
	
	From the above corollary, we can observe that the number of gradient evaluations of $\nabla G$ is still in the order of $\cO((L/\varepsilon)^{1/2})$ under the stochastic setting of $\cH$. The number of stochastic sample evaluation of $\cH$ is bounded by $\cO((L/\varepsilon)^{1/2}+ M/\varepsilon + \sigma^2/\varepsilon^2)$. The improvement of gradient complexity from $\cO(1/\varepsilon^2)$ to $\cO((1/\varepsilon)^{1/2})$ is important when $\nabla G$ is the computational bottleneck for solving the variational inequality. It is interesting to observe that there is now a separation of complexity based on two oracles concerning $\nabla G$ and $\cH$. If we assume that there is a deterministic oracle that returns gradient evaluations of $\nabla G$ for any inquiry point, and that there is a stochastic oracle that returns stochastic sample evaluations of $\cH$ for any inquiry point, our result in the above corollary shows that the complexities for the two oracles should be separated and that we are able to achieve the optimal complexity for each oracle. 
	
	\section{Conclusion}
	In this paper we consider a class of structured monotone variational inequalities over compact feasible sets, in which there exist gradient components in the operators of variational inequalities. We study the research question of whether one can develop numerical algorithms in which the number of gradient evaluations is bounded by $\cO(1/\varepsilon^{1/2})$ when computing an $\varepsilon$-approximate solution. We provide a positive answer to the research question by proposing two algorithms that are able to skip the computations of the gradients from time to time, while still maintaining the optimal iteration complexity for solving the variational inequalities problems. For the deterministic case of variational inequality problems, our proposed mirror-prox sliding method is able to compute an $\varepsilon$-approximate solution with at most $\cO((L/\varepsilon)^{1/2})$ gradient evaluations and $\cO((L/\varepsilon)^{1/2}+M/\varepsilon)$ operator evaluations. For the stochastic case, we also proposed a stochastic mirror-prox sliding method that is able to compute an $\varepsilon$-approximate solution in expectation with at most $\cO((L/\varepsilon)^{1/2})$ gradient evaluations and $\cO((L/\varepsilon)^{1/2}+M/\varepsilon + \sigma^2/\varepsilon^2)$ operator sample evaluations. To the best of our knowledge, our complexity results have not yet been obtained in the literature. Our results also reveals that it is possible to obtain a separation of complexity based on two oracles concerning the gradient evaluation and operator/stochastic operator evaluation, while achieving the optimal complexity for each oracle.
	
	\section{Acknowledgment}
	Both authors are partially supported by the Office of Navel Research grant N00014-20-1-2089.
	
	\bibliographystyle{siam_first_initial}
	\bibliography{yuyuan}

\begin{thebibliography}{10}

\bibitem{auslender2005interior}
{\sc A.~Auslender and M.~Teboulle}, {\em Interior projection-like methods for
  monotone variational inequalities}, Mathematical programming, 104 (2005),
  pp.~39--68.

\bibitem{bregman1967relaxation}
{\sc L.~M. Bregman}, {\em The relaxation method of finding the common point of
  convex sets and its application to the solution of problems in convex
  programming}, USSR computational mathematics and mathematical physics, 7
  (1967), pp.~200--217.

\bibitem{chen1999homotopy}
{\sc X.~Chen and Y.~Ye}, {\em On homotopy-smoothing methods for box-constrained
  variational inequalities}, SIAM Journal on Control and Optimization, 37
  (1999), pp.~589--616.

\bibitem{chen2017accelerated}
{\sc Y.~Chen, G.~Lan, and Y.~Ouyang}, {\em Accelerated schemes for a class of
  variational inequalities}, Mathematical Programming, 165 (2017),
  pp.~113--149.

\bibitem{facchinei2003finite}
{\sc F.~Facchinei and J.-S. Pang}, {\em Finite-dimensional variational
  inequalities and complementarity problems}, vol.~1, Springer, 2003.

\bibitem{jiang2008stochastic}
{\sc H.~Jiang and H.~Xu}, {\em Stochastic approximation approaches to the
  stochastic variational inequality problem}, Automatic Control, IEEE
  Transactions on, 53 (2008), pp.~1462--1475.

\bibitem{juditsky2020statistical}
{\sc A.~Juditsky and A.~Nemirovski}, {\em Statistical Inference via Convex
  Optimization}, Princeton University Press, 2020.

\bibitem{juditsky2011solving}
{\sc A.~Juditsky, A.~Nemirovski, and C.~Tauvel}, {\em Solving variational
  inequalities with stochastic mirror-prox algorithm}, Stochastic Systems, 1
  (2011), pp.~17--58.

\bibitem{korpelevich1976extragradient}
{\sc G.~Korpelevich}, {\em The extragradient method for finding saddle points
  and other problems}, Matecon, 12 (1976), pp.~747--756.

\bibitem{koshal2013regularized}
{\sc J.~Koshal, A.~Nedic, and U.~V. Shanbhag}, {\em Regularized iterative
  stochastic approximation methods for stochastic variational inequality
  problems}, Automatic Control, IEEE Transactions on, 58 (2013), pp.~594--609.

\bibitem{kotsalis2020simplea}
{\sc G.~Kotsalis, G.~Lan, and T.~Li}, {\em Simple and optimal methods for
  stochastic variational inequalities, i: operator extrapolation}, arXiv
  preprint arXiv:2011.02987,  (2020).

\bibitem{kotsalis2020simpleb}
\leavevmode\vrule height 2pt depth -1.6pt width 23pt, {\em Simple and optimal
  methods for stochastic variational inequalities, ii: Markovian noise and
  policy evaluation in reinforcement learning}, arXiv preprint
  arXiv:2011.08434,  (2020).

\bibitem{lan2012optimal}
{\sc G.~Lan}, {\em An optimal method for stochastic composite optimization},
  Mathematical Programming, 133 (1) (2012), pp.~365--397.

\bibitem{lan2021policy}
\leavevmode\vrule height 2pt depth -1.6pt width 23pt, {\em Policy mirror
  descent for reinforcement learning: Linear convergence, new sampling
  complexity, and generalized problem classes}, arXiv preprint
  arXiv:2102.00135,  (2021).

\bibitem{lin2018solving}
{\sc Q.~Lin, M.~Liu, H.~Rafique, and T.~Yang}, {\em Solving
  weakly-convex-weakly-concave saddle-point problems as weakly-monotone
  variational inequality}, arXiv preprint arXiv:1810.10207, 5 (2018).

\bibitem{monteiro2010complexity}
{\sc R.~D. Monteiro and B.~F. Svaiter}, {\em On the complexity of the hybrid
  proximal extragradient method for the iterates and the ergodic mean}, SIAM
  Journal on Optimization, 20 (2010), pp.~2755--2787.

\bibitem{monteiro2011complexity}
\leavevmode\vrule height 2pt depth -1.6pt width 23pt, {\em Complexity of
  variants of {T}seng's modified {F-B} splitting and {K}orpelevich's methods
  for hemivariational inequalities with applications to saddle-point and convex
  optimization problems}, SIAM Journal on Optimization, 21 (2011),
  pp.~1688--1720.

\bibitem{nemirovski2004prox}
{\sc A.~Nemirovski}, {\em Prox-method with rate of convergence ${O}(1/t)$ for
  variational inequalities with {L}ipschitz continuous monotone operators and
  smooth convex-concave saddle point problems}, SIAM Journal on Optimization,
  15 (2004), pp.~229--251.

\bibitem{nemirovski2009robust}
{\sc A.~Nemirovski, A.~Juditsky, G.~Lan, and A.~Shapiro}, {\em Robust
  stochastic approximation approach to stochastic programming}, SIAM Journal on
  Optimization, 19 (2009), pp.~1574--1609.

\bibitem{nesterov2007dual}
{\sc Y.~Nesterov}, {\em Dual extrapolation and its applications to solving
  variational inequalities and related problems}, Mathematical Programming, 109
  (2007), pp.~319--344.

\bibitem{nesterov2018lectures}
{\sc Y.~Nesterov et~al.}, {\em Lectures on convex optimization}, vol.~137,
  Springer, 2018.

\bibitem{nesterov1999homogeneous}
{\sc Y.~Nesterov and J.~P. Vial}, {\em Homogeneous analytic center cutting
  plane methods for convex problems and variational inequalities}, SIAM Journal
  on Optimization, 9 (1999), pp.~707--728.

\bibitem{nesterov1983method}
{\sc Y.~E. Nesterov}, {\em A method for unconstrained convex minimization
  problem with the rate of convergence {$O(1/k^2)$}}, Doklady AN SSSR, 269
  (1983), pp.~543--547.
\newblock translated as Soviet Math. Docl.

\bibitem{ouyang2021lower}
{\sc Y.~Ouyang and Y.~Xu}, {\em Lower complexity bounds of first-order methods
  for convex-concave bilinear saddle-point problems}, Mathematical Programming,
   (2019), pp.~1--35.

\bibitem{rockafellar1976monotone}
{\sc R.~T. Rockafellar}, {\em Monotone operators and the proximal point
  algorithm}, SIAM Journal on Control and Optimization, 14 (1976),
  pp.~877--898.

\bibitem{solodov1999hybrid}
{\sc M.~V. Solodov and B.~F. Svaiter}, {\em A hybrid projection-proximal point
  algorithm}, Journal of convex analysis, 6 (1999), pp.~59--70.

\bibitem{solodov2000inexact}
{\sc M.~V. Solodov and B.~F. Svaiter}, {\em An inexact hybrid generalized
  proximal point algorithm and some new results on the theory of bregman
  functions}, Mathematics of Operations Research, 25 (2000), pp.~214--230.

\bibitem{yousefian2013regularized}
{\sc F.~Yousefian, A.~Nedi{\'c}, and U.~V. Shanbhag}, {\em A regularized
  smoothing stochastic approximation (rssa) algorithm for stochastic
  variational inequality problems}, in Proceedings of the 2013 Winter
  Simulation Conference: Simulation: Making Decisions in a Complex World, IEEE
  Press, 2013, pp.~933--944.

\bibitem{yousefian2014smoothing}
\leavevmode\vrule height 2pt depth -1.6pt width 23pt, {\em On smoothing,
  regularization and averaging in stochastic approximation methods for
  stochastic variational inequalities}, arXiv preprint arXiv:1411.0209,
  (2014).

\bibitem{zhang2021lower}
{\sc J.~Zhang, M.~Hong, and S.~Zhang}, {\em On lower iteration complexity
  bounds for the convex concave saddle point problems}, Mathematical
  Programming,  (2021), pp.~1--35.

\end{thebibliography}
	
\end{document}